\documentclass{amsart}

\usepackage{amsmath, amssymb, amsthm}

\theoremstyle{plain}
\newtheorem{theorem}{Theorem}
\newtheorem{lemma}[theorem]{Lemma}
\newtheorem{corollary}[theorem]{Corollary}
\theoremstyle{remark}
\newtheorem*{example}{Example}
\newtheorem*{acknow}{Acknowledgements}

\newcommand{\abs}[1]{\lvert#1\rvert}
\newcommand\norm[1]{\lVert#1\rVert}
\newcommand\iin{\!\in\!}
\newcommand\conv{\star}
\newcommand\rmd{\mathrm{d}}

\newcommand\pair[1]{\langle\, #1\,\rangle}
\newcommand\ppair[1]{(\,#1\,)}

\newcommand{\lt}{L}
\newcommand{\rt}{R}

\newcommand\sub{\subseteq}

\newcommand{\inv}{^{-1}}

\newcommand\lone{\mathrm{L}^1}
\newcommand\Zlone{\mathrm{ZL}^1}
\newcommand\linfty{\mathrm{L}^\infty}
\newcommand{\luc}{\mathrm{LUC}}
\newcommand{\uc}{\mathrm{U_b}}
\newcommand{\gluc}{G^\mathrm{LUC}}

\newcommand{\M}{\mathrm{M}}
\newcommand{\F}{\mathcal{F}}
\renewcommand{\H}{\mathcal{H}}

\begin{document}

\title[Uniform equicontinuity]%
{Uniformly equicontinuous sets, right multiplier topology,
        and continuity of convolution}

\author{Matthias Neufang}

\address{Carleton University, School of Mathematics and Statistics,
Ottawa, Ontario K1S 5B6, Canada \newline\indent
Universit\'{e} Lille 1 - Sciences et Technologies,
UFR de Math\'{e}matiques,
Laboratoire de Math\'{e}matiques Paul Painlev\'{e} -- UMR CNRS 8524,
59655 Villeneuve d'Ascq Cedex, France}

\email{mneufang@math.carleton.ca, Matthias.Neufang@math.univ-lille1.fr}

\author{Jan Pachl}

\address{Fields Institute, 222 College Street, Toronto, Ontario M5T 3J1, Canada}

\email{jpachl@fields.utoronto.ca}

\author{Pekka Salmi}

\address{Department of Mathematical Sciences, University of Oulu,
        PL~3000, FI-90014 Oulun yliopisto, Finland}

\email{pekka.salmi@iki.fi}

\maketitle


\section{Introduction}
    \label{s:intro}

The dual space of the C*-algebra $\uc(X)$ of
bounded uniformly continuous com\-plex-valued functions on a
uniform space $X$ carries several natural topologies.
One of these is the topology of
uniform convergence on bounded uniformly equicontinuous sets,
or the UEB topology for short.
In the particular case when $X = G$ is a topological group
and $\uc(X) = \luc(G)$, the C*-algebra of
bounded left uniformly continuous functions,
the UEB topology plays a significant role in the continuity
of the convolution product on $\luc(G)^*$.
In this paper we derive a useful characterisation of
bounded uniformly equicontinuous sets on locally compact groups.
Then we demonstrate that for every locally compact group
$G$ the UEB topology on the space $\M(G)$ of finite Radon measures on $G$
coincides with the right multiplier topology, $\M(G)$
viewed as the multiplier algebra of the $\lone$ group algebra.
In this sense the UEB topology is a generalisation to arbitrary
topological groups of the multiplier topology for locally compact
groups. In the final section we prove results about UEB continuity of
convolution on $\luc(G)^\ast$ (even for topological groups not
necessarily locally compact).

Many of our proofs use factorisation techniques such as
Cohen's factorisation theorem (see for example
\cite[Theorem I.11.10]{bonsall-duncan}).
In fact some of our results
may be interpreted as strenghtenings of Cohen's factorisation.
Theorem~\ref{thm:equi-factor} gives a factorisation
of a bounded equiuniformly continuous set of functions in $\luc(G)$
by a single function in $\lone(G)$.
We shall also need the factorisation theorem due to Neufang
\cite{neufang:centre}.


\section{Preliminaries}
\label{s:prel}

Let $X$ be a uniform space and let $\uc(X)$ be the C*-algebra of
bounded uniformly continuous complex-valued functions on $X$.
Say that $\F\sub \uc(X)$ is a \emph{UEB set}
if it is bounded in the sup norm and uniformly equicontinuous.
Say that a net $({m}_\alpha)$ in $\uc(X)^\ast$ converges to $0$
in the \emph{UEB topology} if
\[
\sup_{f\in \F} | \pair{{m}_\alpha,f}| \to 0
\]
for every UEB set $\F\sub \uc(X)$.

For a topological group $G$ (not necessarily locally compact),
$\luc(G)$ is the space of the bounded functions $f$ on $G$ such that
for every $\epsilon>0$ there is a neighbourhood $U$ of
the identity in $G$ with
\[
\abs{f(s) - f(t)}<\epsilon\quad \text{ whenever }st\inv\in U.
\]
Say that a set $\F\sub\luc(G)$ is \emph{equi-$\luc$}
if $\F$ is bounded in the sup norm and
for every $\epsilon>0$ there is a neighbourhood $U$ of
the identity in $G$ such that
\[
\abs{f(s) - f(t)}<\epsilon\qquad \text{whenever }f\in\F,\, st\inv\in U.
\]

It is well known that $\luc(G)$ is the space of the bounded functions
that are uniformly continuous with respect to a suitable uniformity on $G$
(the \emph{right uniformity} in the terminology of
Hewitt and Ross~\cite{hewittRoss:aha}).
In the following we consider $G$ as a uniform space with this uniformity,
so that $\luc(G)=\uc(G)$ and UEB sets are exactly
the equi-$\luc$ sets.

Equi-$\luc$ subsets of $\luc(G)$ arise naturally in the study
of convolution on the dual $\luc(G)^\ast$ of $\luc(G)$.
On general uniform spaces, UEB sets of functions have an important role
in the theory of so called uniform measures ---
see~\cite{pachl:umconvolution}\cite{pachl:umbook} and the references there.

We identify the space $\M(G)$ of finite Radon measures on
a topological group $G$ with a subspace of $\luc(G)^\ast$
by means of the duality $\pair{{m},f}=\int f\, \rmd{m}$
for ${m}\iin\M(G)$ and $f\iin\luc(G)$.
(We shall use the brackets $\pair{\cdot,\cdot}$
to denote Banach space duality in other cases too.)

Denote by $\lt_x f\iin\luc(G)$ and $\rt_x f\iin\luc(G)$
the left and the right translate of $f\iin\luc(G)$
by $x\iin G$; that is,
$\lt_x f(s)=f(xs)$ and $\rt_x f(s)=f(sx)$ for $s\iin G$.
More generally, for ${n}\iin\luc(G)^\ast$,
denote by $\rt_{n} f$ the function $s \mapsto {n}(\lt_s f)$.
Obviously $\rt_{n}=\rt_x$ when ${n}$ is the point mass at $x\iin G$.

Note that $\rt_{n} f\iin\luc(G)$
whenever ${n}\iin\luc(G)^\ast$ and $f\iin\luc(G)$
(this well-known fact is also a special case of Lemma~\ref{lemma:rtranslate}
below).
It follows that we may define the \emph{convolution}
${m}\conv{n}\iin\luc(G)^\ast$
for ${m},{n}\iin\luc(G)^\ast$ by
\[
{m}\conv{n}(f)={m}(\rt_{n} f) \;\text{ for }\; f\iin\luc(G) .
\]
Csisz{\'a}r~\cite{csiszar:conv} and Lau~\cite{lau:tcentre},
among others, have studied this notion of convolution on $\luc(G)^\ast$.
It is also a particular case of the general notion of \emph{evolution}
as defined by Pym~\cite{pym:conv}.
When ${m}, {n} \iin \M(G)$, this definition of ${m}\conv{n}$ agrees
with the standard definition of convolution
of measures~\cite[19.1]{hewittRoss:aha}.
We note that the spectrum $\gluc$ of $\luc(G)$ is closed
under the convolution product of $\luc(G)^*$.

For the remainder of this section, suppose that
$G$ is a locally compact group,
and denote by $\rmd t$ its left Haar measure.
We identify $\lone(G)$ with the measures in $\M(G)$ that
are absolutely continuous with respect to $\rmd t$,
and the convolution in $\luc(G)^\ast$ restricted to $\lone(G)$
agrees with the usual convolution of functions
as defined in \cite[20.10]{hewittRoss:aha}:
\[
f\conv g(s) = \int f(st)g(t\inv)\,\rmd t \qquad
(f,g\iin\lone(G), s\iin G).
\]
The same formula defines also the convolution $f\conv g$ for
$f\iin\lone(G)$ and $g\iin\linfty(G)$ (\cite[20.16]{hewittRoss:aha}).

The dual space of $\lone(G)$ is identified with the
space $\linfty(G)$ of essentially bounded measurable functions on $G$.
The convolution on $\lone(G)$ defines the right action
of $\lone(G)$ on $\linfty(G)$ via
\[
\pair{h\cdot f, g} = \pair{h, f\conv g}\qquad (f,g\in\lone(G), h\in\linfty(G)).
\]
Defining $\widetilde f(s) = f(s\inv)/\Delta(s)$
for $f\iin\lone(G)$, where $\Delta$ is the modular function of~$G$,
we can write $h\cdot f = \widetilde f \conv h$.
It then follows from Cohen's factorisation theorem
that $\luc(G) = \linfty(G)\cdot\lone(G) = \lone(G)\conv\linfty(G)$.

Next we have the left action of $\linfty(G)^\ast$ on $\linfty(G)$:
\[
\pair{m\cdot h, f} = \pair{m, h\cdot f}
\qquad (f\iin\lone(G), h\iin\linfty(G), m\iin\linfty(G)^\ast).
\]
In fact this action depends only on the restriction of
$m$ to $\luc(G)\sub\linfty(G)$ (since $h\cdot f\in\luc(G)$),
so this also gives an action of $\luc(G)^\ast$ on $\linfty(G)$,
which we denote the same way. Moreover, the action leaves $\luc(G)$
invariant. It is easily seen that in fact
$m\cdot h = \rt_m h$ for every $h\in\luc(G)$ and $m\in\luc(G)^\ast$.

Finally, we define the left Arens product on $\linfty(G)^\ast$
by
\[
\pair{n\cdot m, h} = \pair{n, m\cdot h} \qquad
                       (h\in\linfty(G), m,n\in\linfty(G)^\ast).
\]
It follows from the remarks in the preceding paragraph that
$\luc(G)^\ast$ equipped with convolution
is a quotient Banach algebra of $\linfty(G)^\ast$
equipped with the left Arens product.


\section{Characterisation of equi-$\luc$ sets on
         locally compact groups}
    \label{s:characterisation}

The following result characterises equi-$\luc$ sets on
locally compact groups as the bounded sets that can be
simultaneously factorised by one function in $\lone(G)$.

\begin{theorem} \label{thm:equi-factor}
Let $G$ be a locally compact group and let $\F$ be a set
of functions on $G$.
Then $\F$ is equi-$\luc$
if and only if there is $f\in\lone(G)$ and a bounded set
$\H\sub \luc(G)$ such that
\[
\F = f \conv \H.
\]
\end{theorem}

It is easy to check that every set of the form $f\conv\H$ with
$f\in\lone(G)$ and bounded $\H\sub\luc(G)$ is equi-$\luc$,
so one direction of the theorem is clear.

We shall prove the converse separately for compact groups and for
non-compact ones. First we shall derive the compact case from a
general factorisation result for norm-compact subsets of
left Banach $A$-modules. Then we prove the non-compact case
using a different factorisation result
due to Neufang \cite{neufang:centre}. It is noteworthy that
the two factorisation results are based on opposite ideas:
the non-compact case relies heavily on non-compactness
and the compact case on compactness.

We base the compact case to the following theorem. It is actually
a known result due to Craw \cite[Corollary]{craw:factorisation}, but
we present a different proof using Cohen's factorisation theorem
directly. Craw obtained his result as a corollary
to a factorisation theorem for null sequences in Fr\'echet algebras
(similar to Varopoulos's factorisation of null sequences in
Banach algebras). In our case,
the Varopoulos result is an immediate corollary.

\begin{theorem} \label{thm:compact-factor}
Let  $A$ be a Banach algebra and
$X$ a left Banach $A$-module. Suppose that $A$ has
a bounded approximate identity for the action on $X$.
If $K$ is a norm-compact subset of $X$
then there is $a$ in $A$ such that
for every $x$ in $K$ there is $y$ in $X$ such that
\[
x = ay.
\]
\end{theorem}

\begin{proof}
A routine argument shows that
the collection $\mathcal S$ of totally bounded
double sequences on $X$ forms a Banach space.
We define an action of $A$ on $\mathcal S$
coordinatewise:
\[
a(y_{n,m}) = (ay_{n,m}).
\]
The action makes $\mathcal S$ a left Banach $A$-module.
A bounded approximate identity $(e_i)$ for the action of $A$ on $X$ is
also a bounded approximate identity for the action on $\mathcal S$.
Indeed, given $(y_{n,m})\in\mathcal S$ and $\epsilon>0$,
choose $x_1$, \ldots, $x_k$ in $X$ such that
$\epsilon$-balls centred at $x_j$'s cover $(y_{n,m})$.
Then we can choose $i_0$ such that
$\|e_ix_j-x_j\|<\epsilon$ for every $i\ge i_0$ and $j=1$,\ldots, $k$.
Choosing a suitable $x_j$ for each $y_{n,m}$ we have
\[
\|e_iy_{n,m}-y_{n,m}\|\le \|e_iy_{n,m}-e_ix_j\|+\|e_ix_j-x_j\|+\|x_j-y_{n,m}\|
 \le \|e_i\|\epsilon+2\epsilon
\]
for every $i\ge i_0$. So $(e_i)$ is a bounded approximate identity
also for the action on $\mathcal S$.

Since $K$ is norm-compact, there is
a double sequence $(x_{n,m})$ in $K$
such that for every $x$ in $K$,
\[
\|x-x_{n,m}\|<1/n
\]
for some $m$. Then $(x_{n,m})\in \mathcal{S}$, so
we can apply Cohen's factorisation to find $a\in A$ and
$(y_{n,m})\in\mathcal{S}$ such that
\[
x_{n,m} = ay_{n,m}.
\]
Given $x\in K$ there is, by construction,
a sequence $(x_{n,m_n})_{n=1}^\infty$ that converges to $x$.
By factorisation $x_{n,m_n} = ay_{n,m_n}$. The sequence
$(y_{n,m_n})$ is totally bounded and hence has a subsequence
converging to some $y$ in $X$. It follows that $x = ay$.
\end{proof}

\begin{proof}[Proof of the compact case of Theorem~\ref{thm:equi-factor}]
Due to compactness of $G$ an equi-$\luc$ set on $G$ is totally bounded
(i.e.\ precompact), so its closure is a norm-compact subset of $\luc(G)$.
Since $\luc(G)$ is a left Banach $\lone(G)$-module and
$\lone(G)$ has a bounded approximate identity for the action,
we may apply Theorem~\ref{thm:compact-factor}.
\end{proof}

To prove the non-compact case of Theorem~\ref{thm:equi-factor},
it is convenient to use the notation of Arens actions
introduced in Section~\ref{s:prel}.

\begin{proof}[Proof of the non-compact case of Theorem~\ref{thm:equi-factor}]
Let $\F=\{f_i\}$ be an equi-$\luc$ set of functions on a non-compact group $G$.
By Lemma~2.2 of \cite{neufang:centre}, we have a factorisation
\[
f_i = x_i\cdot h
\]
where $h\in\luc(G)$ and $\{x_i\}\sub \gluc$.
By Cohen's factorisation theorem, we can
factorise $h = g\cdot u$ where $u\in\lone(G)$
and $g\in\luc(G)$.
Then
\[
x_i\cdot h = (x_i\cdot g)\cdot u = \widetilde{u} \conv (x_i\cdot g).
\]
Putting $h_i = x_i\cdot g$ and $f = \widetilde{u}$
we get that for every $i$
\[
f_i = f \conv h_i.
\]
This is the required factorisation.
\end{proof}


\section{UEB topology and the right multiplier topology}

In this section we specialise to the case of locally compact
groups and show that on the measure algebra $\M(G)\sub \luc(G)^\ast$
the UEB topology coincides with the right multiplier topology.
Here we view $\M(G)$ as the algebra of right multipliers of $\lone(G)$
(through Wendel's theorem \cite{wendel:multipliers}).
The \emph{right multiplier topology} is the topology
induced by the seminorms $m\mapsto \norm{f \conv m}$ where $f$ runs through
the elements of $\lone(G)$.

\begin{theorem} \label{thm:ueb=rm}
Let $G$ be a locally compact group.
The UEB topology on $\luc(G)^*$
coincides with the topology generated by the seminorms
$m\mapsto \norm{f \conv m}$, $f \in\lone(G)$.
In particular, the right multiplier topology of $\M(G)$
agrees with the UEB topology inherited from $\luc(G)^*$.
\end{theorem}

\begin{proof}
This is rather immediate from the factorisation result
in Theorem~\ref{thm:equi-factor}.
Suppose first that $m_\alpha\to 0$ in $\luc(G)^*$ with respect
to the UEB topology. Let $B$ denote the unit ball of $\luc(G)$.
Then, for every $f$ in $\lone(G)$,
\[
\norm{f \conv m_\alpha} = \sup_{h\in B} \;\abs{\pair{f\conv m_\alpha,h}}
          = \sup_{h\in B} \;\abs{\pair{m_\alpha,\widetilde{f}\conv h}}.
\]
Since $\widetilde{f}\conv B$ is an equi-$\luc$ set,
it follows that $\norm{f\conv m_\alpha} \to 0$.

Conversely, suppose that $m_\alpha\to 0$ in the topology
generated by the seminorms $m\mapsto \norm{f\conv m}$, $f\in\lone(G)$.
Now if $\F$ is an equi-$\luc$ set, then by
Theorem~\ref{thm:equi-factor} we have $\F = f \conv \H$ with
$f\iin\lone(G)$ and $\H\sub \luc(G)$ bounded. Therefore
\[
\sup_{g\in \F} \;\abs{\pair{m_\alpha,g}} =
\sup_{h\in \H} \;\abs{\pair{\widetilde{f}\conv m_\alpha, h}}
\le \norm{\widetilde{f}\conv m_\alpha} \sup_{h\in \H} \norm{h} \to 0.
\]
\end{proof}


\section{Continuity of convolution in $\luc(G)^\ast$}

In this section we prove several results that illustrate the role of
equi-$\luc$ sets in the study of convolution on the dual
$\luc(G)^\ast$ of $\luc(G)$. There are more results of this kind,
presented in a more general framework, in~\cite{pachl:umconvolution}
and~\cite{pachl:umbook}.

\begin{lemma}
    \label{lemma:RMeasures}
Let $G$ be a topological group, ${m}\iin\M(G)$,
and let $\F$ be an equi-$\luc$ subset of $\luc(G)$.
Then the restriction of ${m}$ to $\F$ is continuous in the
$G$-pointwise topology.
\end{lemma}

\begin{proof}
The $G$-pointwise topology and the compact--open topology coincide on every
equi-$\luc$ subset of $\luc(G)$.
\end{proof}

\begin{lemma}
    \label{lemma:semiu}
Let $G$ be a topological group and $\F\sub\luc(G)$ an equi-$\luc$ set.
\begin{enumerate}
\item
    \label{lemma:semiu:right}
The set $\{ \rt_x f \mid f\iin\F, x\iin G \}$ is equi-$\luc$.
\item
    \label{lemma:semiu:left}
For every $x\iin G$ the set $\{ \lt_x f \mid f\iin\F \}$ is equi-$\luc$.
\item
    \label{lemma:semiu:SIN}
If $G$ is a SIN group then the set
$\{ \lt_x f \mid f\iin\F, x\iin G \}$ is equi-$\luc$.
\end{enumerate}
\end{lemma}

\begin{proof}
Part~(\ref{lemma:semiu:right}) follows directly from the definition of
equi-$\luc$ sets.
Part~(\ref{lemma:semiu:left}) follows from the continuity of group
operations: If $U$ is a neighbourhood of the identity in $G$,
then so is $xUx\inv$ for every $x\iin G$.
Part~(\ref{lemma:semiu:SIN}) follows from part~(\ref{lemma:semiu:right})
by the left--right symmetry.
\end{proof}

The next lemma generalises part~(\ref{lemma:semiu:right})
of Lemma~\ref{lemma:semiu}.

\begin{lemma}
    \label{lemma:rtranslate}
Let $G$ be a topological group, $\F\sub\luc(G)$ an equi-$\luc$ set,
and $B\sub\luc(G)^\ast$ a norm-bounded set.
Then the set
$\{ \rt_{n} f \mid f\iin\F, {n}\iin B\}$ is equi-$\luc$.
\end{lemma}

\begin{proof}
Write $c=\sup \{ \norm{{n}} \mid {n}\iin B \}$.
For $f\iin\F$, ${n}\iin B$ and $s\iin G$ we have
\[
\abs{\rt_{n} f (s)} = \abs{{n}(\lt_s f)} \leq \norm{{n}} \cdot \norm{\lt_s f}
= \norm{{n}} \cdot \norm{f} \leq c\norm{f},
\]
so that the set $\{ \rt_{n} f \mid f\iin\F, {n}\iin B\}$ is bounded in
the sup norm.

Take any $\epsilon>0$, and let $U$ be a neighbourhood of the identity in $G$
such that $\abs{f(s) - f(t)}<\epsilon$ whenever $f\iin\F$, $st\inv\iin U$.
If $st\inv\iin U$ and $f\iin\F$, then
\[
\abs{\lt_s f (x) - \lt_t f (x)}  = \abs{f(sx)-f(tx)}
    < \epsilon \;\text{ for }\; x\iin G
\]
and
\[
\abs{\rt_{n} f (s) - \rt_{n} f (t)} = \abs{{n}(\lt_s f - \lt_t f)}
\leq \norm{{n}} \cdot \norm{\lt_s f - \lt_t f} \leq c \epsilon
\]
which completes the proof.
\end{proof}

With the convolution product $\luc(G)^\ast$ is a Banach algebra,
and in particular the convolution is jointly norm-continuous
on $\luc(G)^\ast$. On the other hand, the convolution is far from
being jointly continuous in the $\luc(G)$-weak topology.
By Theorem~1 in~\cite{salmi:jcont},
if $G$ is not compact then the convolution is not jointly $\luc(G)$-weakly
continuous even when restricted to bounded sets in $\M(G)$.

Nevertheless, we now show that the convolution on $\M(G)$ is well-behaved
in the UEB topology.
Recall from section~\ref{s:intro} that this is the topology on $\luc(G)^\ast$
of uniform convergence on the equi-$\luc$ subsets of $\luc(G)$.

The proof of the following theorem does not use any properties
of Radon measures other than Lemma~\ref{lemma:RMeasures}.
Thus the same proof establishes a more general result in which $\M(G)$
is replaced by the space of uniform measures.
For this and other generalisations,
see~\cite{pachl:umconvolution}\cite{pachl:umbook}.

\begin{theorem}
    \label{th:cconv}
Let $G$ be a topological group, let $B\sub\luc(G)^\ast$ be a norm-bounded set,
${m}_0\iin\luc(G)^\ast$ and ${n}_0\iin B$.
Then the mapping $\ppair{{m},{n}}\mapsto{m}\conv{n}$ from $\luc(G)^\ast\times B$
to $\luc(G)^\ast$ is jointly continuous at $\ppair{{m}_0,{n}_0}$
in the UEB topology in each of these two cases:
\begin{itemize}
\item[\textit{(a)}]
${m}_0\iin\M(G)$,
\item[\textit{(b)}]
$G$ is a SIN group.
\end{itemize}
\end{theorem}

For commutative groups, a result similar to case (a) (for uniform
measures) was proved by Caby~\cite{caby:conv}.
Other theorems of this kind were also proved by
Csisz{\'a}r~\cite[Th.1]{csiszar:conv} and Pym~\cite[2.2]{pym:conv65}.
(There is a gap in the proof of Theorem~1 in~\cite{csiszar:conv}; to
correct it, a boundedness condition should be added to the statement
of the theorem.)

\begin{proof}
Let $({m}_\alpha)_{\alpha\in A}$ be a net in $\luc(G)^\ast$ and
$({n}_\alpha)_{\alpha\in A}$ a net in $B$.
Let ${m}_\alpha \to {m}_0$ and ${n}_\alpha \to {n}_0$ in the UEB topology.
Take any equi-$\luc$ set $\F\sub\luc(G)$.

Define
$g_\alpha(x)=\sup \{\abs{\rt_{{n}_\alpha}{f}(x) - \rt_{{n}_0}{f}(x)}
\mid {f\in\F}\}$
for $\alpha\iin A$, $x\iin G$.
It follows from Lemma~\ref{lemma:rtranslate} that
the set $\{g_\alpha\mid \alpha\iin A\}$ is equi-$\luc$.
On the other hand, for every $x\iin G$ the set
$\{ \lt_x f \mid f\iin\F \}$ is equi-$\luc$ by Lemma~\ref{lemma:semiu},
and since ${n}_\alpha \to {n}_0$ in the UEB topology, $g_\alpha(x) \to 0$.

Now
\begin{align*}
&\abs{{m}_\alpha\conv{n}_\alpha(f) - {m}_0\conv{n}_0(f)}
 = \abs{{m}_\alpha(\rt_{{n}_\alpha}{f}) - {m}_0(\rt_{{n}_0}{f})}           \\
&\qquad \leq \abs{{m}_\alpha(\rt_{{n}_\alpha}{f}) - {m}_0(\rt_{{n}_\alpha}{f})}
 + \abs{{m}_0(\rt_{{n}_\alpha}{f}) - {m}_0(\rt_{{n}_0}{f})}.
 \end{align*}
The set $\{\rt_{{n}_\alpha}{f}\mid f\iin\F, \;\alpha\iin A\}$
is equi-$\luc$ by Lemma~\ref{lemma:rtranslate}, so
\[
\lim_\alpha \;\sup_{f\in\F} \;
\abs{{m}_\alpha(\rt_{{n}_\alpha}{f}) - {m}_0(\rt_{{n}_\alpha}{f})} = 0.
\]
Therefore,
\begin{align*}
\lim_\alpha \;\sup_{f\in\F} \;
\abs{{m}_\alpha\conv{n}_\alpha(f) - {m}_0\conv{n}_0(f)}
& \leq \lim_\alpha \;\sup_{f\in\F} \;
\abs{{m}_\alpha(\rt_{{n}_\alpha}{f}) - {m}_0(\rt_{{n}_\alpha}{f})} \\
& \quad + \lim_\alpha \;\sup_{f\in\F} \;
\abs{{m}_0(\rt_{{n}_\alpha}{f}) - {m}_0(\rt_{{n}_0}{f})} \\
& \leq \lim_\alpha
\;\abs{{m}_0} ( g_\alpha )
= 0
\end{align*}
where the last equality holds by Lemma~\ref{lemma:RMeasures}
in case~(a).
In case~(b) we have $\norm{g_\alpha} \to 0$ by part~\ref{lemma:semiu:SIN}
of Lemma~\ref{lemma:semiu}, and thus again $\abs{{m}_0} ( g_\alpha ) \to 0$.
\end{proof}

\begin{corollary}
    \label{cor:jointMG}
Convolution is jointly UEB continuous on bounded subsets of $\M(G)$
for every topological group $G$.
\end{corollary}

\begin{corollary}
    \label{cor:jointSIN}
Convolution is jointly UEB continuous on bounded subsets of $\luc(G)^\ast$
for every SIN group $G$.
\end{corollary}

It is known~\cite{pachl:umconvolution}\cite{pachl:umbook}
that in $\M(G)$ the UEB topology and the $\luc(G)$-weak topology
coincide on the positive cone, on $\norm{\cdot}$ spheres and on
$\luc(G)$-weakly compact sets.
Thus from Corollary~\ref{cor:jointMG} we immediately obtain
that the convolution is jointly $\luc(G)$-weakly continuous
when restricted to the positive cone or a $\norm{\cdot}$ sphere
or a $\luc(G)$-weakly compact set in $\M(G)$.
In particular, $\conv$ is jointly $\luc(G)$-weakly sequentially
continuous on $\M(G)$.

\begin{example}
In this example we exhibit a topological group $G$
such that the convolution is not jointly UEB continuous on bounded subsets of
$\luc(G)^\ast$. In particular, the statement of  Theorem~\ref{th:cconv}
is not true if both conditions (a) and (b) are removed.
In fact, for this group $G$ there are ${m}_0\iin\gluc$ and a sequence
of elements $y_j\iin G$ that UEB converges to $y\iin\gluc$, and yet
${m}_0\conv y_j$ does not even converge to ${m}_0\conv y$ in the
$\luc(G)$-weak topology.

Let $G$ be the group of increasing homeomorphisms of the interval
$[0,3]$ onto itself.
The group operation is composition of maps.
The topology of $G$ is induced by the right-invariant metric
\[
\Delta(x,y):= \sup_{t\in[0,3]} \abs{x(t)-y(t)} \;\text{ for }\; x,y\iin G.
\]

Define $y_j\iin G$ for $j=1,2,\dotsc$ by
\[
y_j(t) := \left\{ \begin{array}{lll}
                     t/j    & \text{when} & 0\leq t \leq 2  \\
                     \frac{2}{j} + (3 - \frac{2}{j})(t-2) &
                     \text{when} & 2 < t \leq 3.
                  \end{array} \right.
\]
The sequence $\{y_j\}_j$ is Cauchy in the metric,
hence it has a UEB limit $y\iin G^\luc$.

Define the function $h\iin\luc(G)$ by $h(x):=x(1)$ for $x\iin G$.
Let ${m}_0$ be a cluster point of the sequence $(y_k\inv)$ in $\gluc$.
Then
\[
h(y_k\inv \circ y_j) = y_k\inv(y_j(1)) = y_k\inv(1/j)
                    = 2 + \frac{\frac{1}{j} - \frac{2}{k}}{3-\frac{2}{k}}
\]
for all $k\geq 2j$, and so
\[
{m}_0 \conv y_j (h) = \lim_k h(y_k\inv \circ y_j) = 2 + 1/3j.
\]
On the other hand,
\[
\rt_y h(x) = y(\lt_x h) = \lim_j \lt_x h(y_j) = \lim_j h(x\circ y_j) =
            \lim_j x(y_j(1))  = 0
\]
for all $x\iin G$, and so
\[
{m}_0 \conv y (h)      = {m}_0(\rt_y h) = 0.
\]
Hence ${m}_0\conv y(h) \neq \lim_j {m}_0\conv y_j(h)$.
\qed
\end{example}

We do not know whether Corollary~\ref{cor:jointMG} holds without the
boundedness restriction; that is, whether convolution is jointly UEB
continuous on $\M(G)$ for every topological group $G$.
However, we now show that convolution is jointly UEB continuous even
on $\luc(G)^\ast$ for every locally compact SIN group $G$.
The proof relies on the characterisation of equi-$\luc$ sets
in section~\ref{s:characterisation}.

\begin{theorem}
    \label{th:cconvLCA}
For every locally compact SIN group $G$,
convolution is jointly UEB continuous on $\luc(G)^\ast$.
\end{theorem}

\begin{proof}
By Theorem~\ref{thm:ueb=rm}
it is enough to show that the convolution on
is jointly continuous with respect to the
topology generated by the seminorms $m\mapsto\norm{f\conv m}$,
$f\in\lone(G)$.
Let $m_\alpha\to m$ and $n_\alpha\to n$ be nets
converging in this topology.

Let $\Zlone(G)$ denote the centre of $\lone(G)$.
Then $\Zlone(G)$ consists of all central functions in $\lone(G)$
(i.e.\ functions $f\in\lone(G)$
such that $f(st) = f(ts)$ for almost every $s,t\in G$).
Since $\lone(G)$ is weak*-dense in $\luc(G)^*$ and
the convolution from left by an element in $\lone(G)$
is weak*-continuous on $\luc(G)^*$, it follows
that $\Zlone(G)$ is contained in the centre of $\luc(G)^*$ as well.
Now $\lone(G)$ has a bounded approximate
identity consisting of central functions because $G$ is SIN.
Hence by Cohen factorisation theorem, $\lone(G) = \lone(G)\conv\Zlone(G)$.

Let $f\iin\lone(G)$ be arbitrary and factorise $f = g\conv h$
where $g\iin\lone(G)$, $h\iin\Zlone(G)$.
Then
\[
f\conv (m_\alpha\conv n_\alpha) =
(g\conv m_\alpha)\conv (h\conv n_\alpha)
\]
because $h\in \Zlone(G)$. Since $g\conv m_\alpha\to g\conv m$ and
$h\conv n_\alpha\to h\conv n$ in norm,
it follows that
$f\conv (m_\alpha\conv n_\alpha) \to f\conv (m\conv n)$ in norm,
as required.
\end{proof}

\begin{acknow}
Jan Pachl is grateful for being able to work in the supportive environment
at the Fields Institute.
Pekka Salmi thanks Nico Spronk for support during a postdoctoral stay
at the University of Waterloo.
Later support by Emil Aaltonen Foundation is also gratefully acknowledged.
\end{acknow}

\end{document}